\documentclass{amsart}
\usepackage{graphicx, amsmath, amsthm, amsfonts, amssymb,color, theoremref, hyperref, float} 
\usepackage{a4wide}
\usepackage{enumerate}
\usepackage{comment}
\usepackage[backend=bibtex,style=numeric]{biblatex}

\DeclareMathOperator{\Val}{\operatorname{Val}}

\DeclareMathOperator{\Sym}{\operatorname{Sym}}

\DeclareMathOperator{\diam}{\operatorname{diam}}

\DeclareMathOperator{\tr}{\operatorname{tr}}
\DeclareMathOperator{\conj}{\operatorname{conj}}
\DeclareMathOperator{\Ad}{\operatorname{Ad}}

\newtheorem{theorem}{Theorem}[section]
\newtheorem{lemma}[theorem]{Lemma}
\newtheorem{proposition}[theorem]{Proposition}

\newtheorem{definition}[theorem]{Definition}
\newtheorem{remark}[theorem]{Remark}

\newtheorem*{theorem*}{Theorem}
\newtheorem*{lemma*}{Lemma}
\newtheorem*{proposition*}{Proposition}
\newtheorem*{corollary*}{Corollary}
\newtheorem*{definition*}{Definition}
\newtheorem*{remark*}{Remark}

\begin{document}
\title[Kinematic formulas in convex geometry]{Kinematic formulas in convex geometry  \\ for non-compact groups}

\author[S. Anjos]{S\'ilvia Anjos}
\email{sanjos@math.ist.utl.pt}

\author[F. Nascimento]{Francisco Nascimento}
\email{francisco.d.nascimento@tecnico.ulisboa.pt}

\address{Center for Mathematical Analysis, Geometry and Dynamical Systems \\ Department of Mathematics \\  Instituto Superior T\'ecnico \\ Universidade de Lisboa \\ Portugal.}

\begin{abstract}
We generalize classical kinematic formulas for convex bodies in a real vector space $V$ to the setting of non-compact Lie groups admitting a Cartan decomposition. Specifically, let \( G \) be a closed linear group with Cartan decomposition \( G \cong K \times \exp(\mathfrak{p}_0) \), where \( K \) is a maximal compact subgroup acting transitively on the unit sphere. For \( K \)-invariant continuous valuations on convex bodies, we establish an integral geometric-type formula (\thref{Main}) for $\overline{G} = G \ltimes V$.
Key to our approach is the introduction of a Gaussian measure on \( \mathfrak{p}_0 \), which ensures convergence of the non-compact part of the integral. In the special case \( K = O(n) \), we recover a Hadwiger-type formula (Theorem \ref{thm: O(n) case}) involving intrinsic volumes, with explicit constants \( c_j \) computed via a Weyl integration formula. 
\end{abstract}


\maketitle

\section{Introduction}

Let $(V, \left< \cdot, \cdot\right>)$ be an $n$-dimensional real inner product space and let $\mathcal{K}\left( V \right) $ denote the space of convex bodies in $V$, i.e., non-empty convex compact subsets of $V$. A valuation is a function $\Phi: \mathcal{K}\left( V \right) \to \mathbb{R}$ such that \[
\Phi\left( M \cup L \right) = \Phi\left( M \right) + \Phi\left( L \right) - \Phi\left( M \cap  L \right)  
\]  
whenever $M,L, M \cup L \in \mathcal{K}(V)$. Valuations are a classic concept of convex geometry with applications in integral geometry; see \textcite{schneider2013convex}, \textcite{klain1997introduction} and, more recently, \textcite{hug2020lectures}. They generalize classical notions like volume and Euler characteristic, enabling unified proofs of integral geometric formulas \\
One fundamental result that shows the connection between valuations and classical integral geometry is \textit{Steiner's Formula}, which states that the volume of an $\varepsilon$-thickening of a convex body $M$ can be written as a polynomial in $\varepsilon$, that is,
\begin{align}\label{Steiner's Formula}
	V_n\left( M+\varepsilon B^n \right)= \sum_{j=0}^{n} \varepsilon^{n-j}\kappa_{n-j}V_{j}\left( M \right)  
,
\end{align}
where $\kappa_j$ is the  $j$-th dimensional volume of the $j$-th dimensional unit ball $B^{j}$. This result goes back to \textcite{steiner1840parallele}, but nowadays it is a particular case of a more general expansion formula using mixed volumes (see \textcite[Chapter 5]{schneider2013convex}).  
Here, $V_{n}$ is the volume in $V$, and more generally, we call the functions $V_j$ the \textit{intrinsic volumes}. So this formula describes how the volume of a thickened convex body decomposes into contributions from its intrinsic volumes, which encode geometric properties like surface area and mean curvature.

One can show (see, e.g., \textcite[Theorem 5.5]{schneider1999integral}) that they are \textit{translation-invariant}, $O(n)$\textit{-invariant} valuations, which means that \[
V_j\left( M +t \right)=V_j\left( M \right), \quad \forall t \in V, 
\]
and \[
V_j\left( \vartheta M \right) =V_j\left( M \right), \quad \forall \vartheta \in O\left( n \right),  
\]
respectively. Furthermore, they are continuous with respect to the \textit{Hausdorff metric}, which is defined by \[
	\delta\left( M,L \right) =\min\{\varepsilon \ge 0: M \subset L + \varepsilon B^n, \ L \subset M + \varepsilon B^n\}
,\]
and for each $j$,  $V_j$ is $j$\textit{-homogeneous}, that is, \[
V_j\left( \lambda M \right) =\lambda^{j}V_j\left( M \right), \quad \forall \lambda \geq 0, \ \forall M \in \mathcal{K}\left( V \right) 
.\] 
All these properties are very important in the modern theory of valuations, so a more compact notation has been established. Let $\Val(V)$ be the continuous, with respect to the Hausdorff metric, translation-invariant valuations, and $Val^{K}\left( V \right) $ be the $K$-invariant valuations for some group $K$. We also let \[
\Val_j \left(V \right)= \left\{\varphi \in \Val(V): \varphi(\lambda M)=\lambda^{j} \varphi(M),  \ \forall \lambda \geq 0, \ \forall M \in \mathcal{K}(V) \right\}
\] 
be the space of $j$-homogeneous valuations. Thus, we  have $V_j \in \Val_{j}^{O(n)}\left( V \right) $ for each $j$. The remarkable fact observed by Hadwiger is that all $O(n)$-invariant valuations can be written as linear combinations of intrinsic volumes.
\begin{theorem}[Hadwiger's Characterization Theorem \protect\cite{hadwiger2013vorlesungen}]\thlabel{Hadwiger's Characterization Theorem}
	Let $V$ be a real $n$-dimensional vector space. Then\[
	\Val^{O(n)}\left( V \right) = \Val^{SO(n)}\left( V \right)=  \left<V_0,\ldots,V_n \right>
	.\] 
\end{theorem}

We remind the reader of some notation. Let $K$ be a compact group acting linearly on $V$ and let $\overline{K} \cong K \ltimes V$ be the associated group of affine transformations. Consider an Euclidean inner product on $V$, and let $f$ be a linear map that sends the canonical basis of $\mathbb{R}^n$ to some orthonormal basis of $V$. We equip $K$ with a Haar probability measure $\nu_K$ and $\overline{K}$ with a Haar measure $\mu_{\overline{K}}$ normalized so that $\left\{ \overline{k}=t_x \circ k: k\in K, \, x \in f\left(\left[ 0,1 \right]^{n}\right) \right\}=1 $, where $t_x$ denotes translation by $x$. We denote $Gr_j\left( V \right) $ the Grassmannian of all $j$ dimensional vector subspaces of the real vector space $V$ and $\overline{Gr}_{j}\left( V \right) $ the corresponding space of affine subspaces. Using the finite-dimensionality of $\Val^{SO(n)}\left( V \right) $, Hadwiger obtained the following kinematic formula.
\begin{theorem}[Hadwiger's General Integral Geometric Theorem \protect\cite{hadwiger2013vorlesungen}]\thlabel{Hadwiger's General Integral Geometric Theorem}
	If $\varphi: \mathcal{K}\left( V \right)  \to \mathbb{R}$ is a continuous valuation then 
    \begin{align}\label{kinematic formula}
        \int_{\overline{SO(n)}} \varphi(M \cap \overline{k}L) d\mu_{\overline{SO(n)}}(\overline{k})= \sum_{j=0}^{n}\varphi_{n-j}(M)V_j(L),
    \end{align}
    for $M,L \in \mathcal{K}\left( V \right) $, where the coefficients $\varphi_{n-j}$ are given by
    \begin{align}\label{coefficient of kinematic formula}
    \varphi_{n-j}(M)=\int_{\overline{Gr}_{j}\left( V \right) } \varphi(M \cap E) d \mu_{j}(E),
    \end{align}
   and  where $\mu_j$ is a $SO(n)$-invariant measure, normalized so that \[
    \mu_j\left( \left\{ E \in \overline{Gr}_j\left( V \right) : E \cap B^{n} \right\}  \right) = \kappa_{n-j}
    .\] 
\end{theorem}
The idea of the proof of \thref{Hadwiger's General Integral Geometric Theorem} is showing that $\Phi\left( L \right):=\int_{\overline{SO(n)}} \varphi(M \cap \overline{k}L) d\mu_{\overline{SO(n)}}(\overline{k}) $ is in $\Val^{SO(n)}(V)$. Then, using \thref{Hadwiger's Characterization Theorem}, we can write $\Phi\left( L \right) $ as in \eqref{kinematic formula}, and we can find $\varphi_{n-j}(M)$ by evaluating at appropriate $j$-dimensional cubes (see \textcite[Theorem 7.1]{schneider1999integral}).  

For a general compact group $K$, by using the same techniques as in the proof of \thref{Hadwiger's General Integral Geometric Theorem} and further assuming that $\phi \in \Val^{K}\left( V \right) $, one can also obtain other types of kinematic formulas. 

\begin{theorem}[{\textcite[Section 4.2]{bernig2011algebraic}}]\thlabel{Classical General kinematic formula}
    Let $K$ be a compact group such that $\dim \Val^{K}\left( V \right)  < \infty$, and let $\psi_1,\dots,\psi_N$ be a basis of $\Val^K\left( V \right) $. Given $\phi \in \Val^{K}(V)$, then there are constants $c^{\Phi}_{ij}$ such that, if $M,L \in \mathcal{K}\left( V \right)$ then 
    \begin{align}\label{Bernig kinematic formula}
        \int_{\overline{K}} \phi(M \cap \overline{k} L) d\mu_{\overline{K}}(\overline{k}) =\sum_{i=1}^{N} \sum_{j=1}^{N} c_{i j}^{\phi} \psi_i(M) \psi_j(L).
    \end{align}
\end{theorem}

This begs the question as for what $K$ is the space $\Val^{K}$ finite-dimensional, and to explicitely find the $c_{ij}^{\phi}$. The first part was answered by Alesker: 

\begin{theorem}[{\cite[Proposition 2.6]{alesker2007theory}}]\thlabel{Alesker finite-dimensionality theorem}
	Let $K$ be a compact subgroup of $O(n)$. The space $\Val^{K}\left( V \right) $ is finite-dimensional if and only if $K$ acts transitively on the unit sphere of $V$.
\end{theorem}
This result motivates our focus on groups with transitive sphere actions, as their valuation spaces are tractable. 
The connected compact Lie groups that act transitively on the sphere have been classified by \textcite{borel1949some,borel1950plan} and \textcite{montgomery1943transformation}. There are 6 infinite series \[
SO(n),U(n),SU(n),Sp(n),Sp(n)\cdot Sp(1),Sp(n)\cdot U(1)
,\]
and 3 exceptions \[
	G_2, \text{Spin}\left( 7 \right) , \text{Spin}\left( 9 \right) 
.\] 
Moreover, the following result of McMullen tells us that we can consider a basis of $\Val^{K}(V)$ that consists of homogeneous valuations.
\begin{theorem}[\cite{mcmullen1977valuations}]\thlabel{McMullen's Theorem}
	Let $V$ be a n-dimensional real vector space. Then we have \[
	\Val\left( V \right) = \oplus_{j=0}^{n}Val_{j}\left( V \right) 
	.\] 
\end{theorem}
In particular, $\Val^{K}\left( V \right) =\oplus_{j=0}^{n}Val^{K}_{j}\left( V \right) $. For example, for $\Val^{SO(n)}\left( V \right) $, each $\Val^{SO(n)}_{j}\left( V \right) = \left< V_{j} \right>$ is one-dimensional, but for general $K$ their dimensions can be greater. In fact, in \cite{alesker2002hard} Alesker found a basis for the case $K=U(n)$. Using a different basis, Berning and Fu found in \cite{bernig2011hermitian}  the coefficients for the kinematic formula \eqref{Bernig kinematic formula} when the valuation $\phi$ is the Euler Characteristic $\chi$, and more generally derived a method to compute kinematic formulas in the $U(n)$ case. \\

All of the above results involve compact groups. In this paper, we will study kinematic formulas for real, non-compact, closed linear groups $G$ which admit a Cartan Decomposition. The setting will be as follows: letting $\mathfrak{g}_0$ be the respective Lie algebra, we write it as the direct sum $\mathfrak{g}_0 =\mathfrak{k}_0 \oplus \mathfrak{p}_0$, where $\mathfrak{k}_0$ (resp. $\mathfrak{p}_0$) is the $+1$ (resp. $-1$) eigenspace with respect to a Cartan involution $\theta$. More precisely, $\theta$ is an involutive Lie algebra automorphism $\theta: \mathfrak{g}_0 \to \mathfrak{g}_0$ (i.e., $\theta^2 = {\rm id}$) such that the bilinear form 
$$ B_\theta (X,Y) := - B(X, \theta Y)$$
is positive definite, where $B$ is the Killing form on $\mathfrak{g}_0$. Recal that this bilinear form is given by 
$$B(X,Y) = {\rm Tr} ( {\rm ad} (X) \circ {\rm ad} (Y)),$$
where ${\rm ad}: \mathfrak{g}_0 \to \mathfrak{g}_0 $ is the adjoint representation and ${\rm Tr}$ is the trace of a linear operator.

Through $\theta$, we obtain a diffeomorphism $G \cong K \times \exp\left( \mathfrak{p}_0 \right) $, where $K=\left\{ x \in G: \Theta x = x \right\} $ (see \thref{Cartan Decomposition} below), where $\Theta: G \to G$ is an involutive automorphism such that its differential is $\theta$. If $K$ is in the conditions of \thref{Alesker finite-dimensionality theorem}, then by using a Gaussian measure $\gamma_{\mathfrak{p}_0}$ (see \eqref{GaussianMeasure} for the definition) to integrate along $\mathfrak{p}_0$, we generalize \thref{Classical General kinematic formula} to non-compact groups. First we set some notation. Let $\lambda_V$ be the Lebesgue measure of the vector space $V$ and define the product measure on $\overline{G} := G \ltimes V $ by $ m_{\overline{G}}: = \nu_{K} \times \gamma_{\mathfrak{p}_0} \times \lambda_V $.

\begin{theorem}\thlabel{Main}
	Let $G \cong K \times \exp\left( \mathfrak{p}_0 \right) $ be a closed linear group which admits a Cartan Decomposition with respect to $\Theta$. Let $\phi \in \Val^{K}\left( V \right) $ and $M,L \in \mathcal{K}\left( V \right) $. If $K$ is such that $\dim \Val^{K}\left( V \right) < \infty$, then we have
	\begin{align}\label{Main formula}
        &\int_{\overline{G}} \phi( M \cap \overline{g} L) \, dm_{\overline{G}} (\overline{g})
        =\sum_{i=1}^{N} \sum_{j=1}^{N}  \sum_{\psi_l \in Val_{\deg\left( \psi_j \right) }^{K}\left( V \right) }c_{ij}^{\phi} \, c\left( \psi_l \right) \psi_i(M)  \psi_l(L),
    \end{align}
    where $c_{ij}^{\phi}$ and $c\left( \psi_l \right) $ are constants.
\end{theorem}

An analagous result, for the case $K=O(n)$, is a generalization of \thref{Hadwiger's General Integral Geometric Theorem}:

\begin{theorem}\label{thm: O(n) case}
	Let $M,L \in \mathcal{K}\left( V \right) $. For $\varphi:\mathcal{K}\left( V \right)  \to \mathbb{R}$ a continuous valuation, we have that
	\begin{align*}
	\int_{\overline{GL(n)}} \varphi( M \cap \overline{g} L) \, dm_{\overline{GL(n)}} (\overline{g})  =2\sum_{j=0}^{n} c_j \, \varphi_{n-j}(M) V_j(L),
	\end{align*}
	where $\varphi_{n-j}$ is given by
	\begin{align*}
	\varphi_{n-j}(M)=\int_{\overline{Gr}_{j}} \varphi(M \cap E) d \mu_{j}(E)
	\end{align*}
	and $c_j$ is given by 
    \begin{equation}\label{coefficients in SO(3) case}
	c_j=\frac{\int_{\Sym(n)} V_j(e^{X} B^n) d\gamma_{\Sym(n)}(X)}{\binom{n}{j}\frac{\kappa_n}{\kappa_{n-j}}}.
	\end{equation}
\end{theorem}
 
 \thref{Main} can be applied, for example, to the following list of Lie groups (see \cite[Chapter X, Section 3]{helgason1979differential}). Although \cite{helgason1979differential} considers only the case of connected Lie groups, our methods can be applied also to $O(n)$ as described in the proof of Theorem \ref{thm: O(n) case} (see also Remark \ref{2 components}). Note that different groups $G$ for the same compact group $K$ means that a different Cartan decomposition $\theta$ is being used.
 \begin{table}[H]
     \centering
     \begin{tabular}{|c|c|} \hline 
        $K$ & $G$  \\ \hline
        $SO(n)$ & $SL(n)$\\ \hline 
         $O(n)$ & $GL(n)$\\ \hline 
         $Sp(n)$ & $SU^{\ast}(2n)$ \\ \hline 
         $U(n)$ & $Sp(n,\mathbb{R})$\\ \hline 
         $U(n)$ & $SO^*(2n)$\\ \hline 
         $U(n)$ & $GL(n,\mathbb{C})$\\ \hline 
         $SU(n)$ & $SL(n,\mathbb{C})$\\ \hline
     \end{tabular}
     \caption{Non-exhaustive list of groups where \thref{Main} is applicable.}
     \label{tab:group list}
 \end{table}

\medskip 
\noindent {\bf Organization of the paper.} In Section \ref{Preliminaries}, we present the technical background necessary to prove the main theorems. Section \ref{Proof} presents the proof of \thref{Main}, while Section \ref{O(n) case} is devoted to analyzing the case $K=O(n)$ and therefore proves Theorem \ref{thm: O(n) case}. At the end of this section, we also provide a method for computing the coefficients $c_j$ in the theorem.

\medskip 
\noindent {\bf Acknowledgments.} We are grateful to Dmitry Faifman, Gustavo Granja, Daniel Hug, Naichung Conan Leung, José Natário, João Pimentel Nunes and Martin Pinsonnault for their helpful conversations and useful comments. The first author is partially supported by FCT/Portugal through project 2023.13969.PEX, with DOI identifier 10.54499/2023.13969.PEX, and both authors acknowledge support from FCT/Portugal through project UIDB/04459/2020, with DOI identifier 10.54499/UIDP/04459/2020.

\section{Preliminaries} \label{Preliminaries}

In this section, we introduce some technical concepts that will be used in the proof of \thref{Main}.

\subsection{Cartan Decomposition}
We remind the reader of the Cartan Decomposition on the Lie algebra level.

\begin{definition}
    Let $\mathfrak{g}_0$ be a semisimple Lie algebra over $\mathbb{R}$. A vector space direct sum \[
    \mathfrak{g}_0=\mathfrak{k}_0 \oplus \mathfrak{p}_0
    \]
    is called a Cartan Decomposition if we have the relations \[
    [\mathfrak{k}_0,\mathfrak{k}_0] \subset \mathfrak{k}_0, [\mathfrak{k}_0,\mathfrak{p}_0] \subset \mathfrak{p}_0, [\mathfrak{p}_0,\mathfrak{p}_0] \subset \mathfrak{k}_0
    \]
    and the Killing form $B_{\mathfrak{g}_0}$ is negative definite on $\mathfrak{k}_0$ and positive definite on $\mathfrak{p}_0$.
\end{definition}
A Cartan Decomposition determines an involution $\theta$ by the formula \[
\theta = \left\{ \begin{array}{cl}
+ {\rm id} &  \ \text{on } \mathfrak{k}_0 \\
- {\rm id} &  \ \text{on } \mathfrak{p}_0
\end{array} \right.
\]
Conversely, an involution $\theta$ such that $B_\theta(X,Y)=-B(X,\theta Y)$ is positive definite gives a Cartan Decomposition through its $+1$ and $-1$ eigenspaces. We call such $\theta$ a \textit{Cartan involution}. \\
We can extend this decomposition to the level of Lie groups.
\begin{theorem}[{\cite[Theorem 6.31]{knapp1996lie}}]\thlabel{Cartan Decomposition}
    Let $G$ be a semisimple Lie group, let $\theta$ be a Cartan involution of its Lie algebra $\mathfrak{g}_0$, let $\mathfrak{g}_0=\mathfrak{k}_0 \oplus \mathfrak{p}_0$ be the corresponding Cartan decomposition, and let $K$ be a connected Lie subgroup of $G$ with Lie algebra $\mathfrak{k}_0$. Then 
    \begin{enumerate}[(a)]
        \item there exists a Lie group automorphism $\Theta$ of $G$ with differential $\theta$, and $\Theta$ has $\Theta^2=1$;
        \item the subgroup of $G$ fixed by $\Theta$ is $K$;
        \item the mapping $K \times \mathfrak{p}_0 \to G$ given by $(k,X) \mapsto k \exp X$ is a diffeomorphism onto;
        \item $K$ is closed;
        \item $K$ contains the center $Z$ of $G$;
        \item $K$ is compact if and only if $Z$ is finite;
        \item when $Z$ is finite, $K$ is a maximal compact subgroup of $G$.
    \end{enumerate}
\end{theorem}
\subsection{Gaussian measure on \texorpdfstring{$\mathfrak{p}_0$}{p0}}
Keeping in mind that we  want to integrate along a non-compact group $G$ with a decomposition as in \thref{Cartan Decomposition}, we have to find a suitable finite measure in $\mathfrak{p}_0$ such that the left hand side of \eqref{Main formula} is a finite value.
Let $d=\dim \mathfrak{p}_0$. We remind the reader that the Frobenius inner product is defined by $\left< X,Y \right>_{F}=\tr\left( X^{T}Y \right) $. It is clear that $\langle \cdot,\cdot \rangle_F$ is $O(n)$-invariant, that is, \[
\left \langle \vartheta X,\vartheta Y\right \rangle_F=\left \langle X\vartheta,Y\vartheta\right \rangle_F=\left \langle X,Y\right \rangle_F
\]
for $\vartheta \in O(n)$. In what follows, we equip $\mathfrak{p}_0$ with the Frobenius inner product. Consider the linear map $h:\mathbb{R}^{d} \to \mathfrak{p}_0$ that sends the canonical basis of $\mathbb{R}^{d}$ to an orthonormal basis of $\mathfrak{p}_0$. By construction, it is a unitary transformation. We define the standard Gaussian measure $\gamma_{\mathfrak{p}_0}$ on $\mathfrak{p}_0$ as 
\begin{align}\label{GaussianMeasure}
    \gamma_{\mathfrak{p}_0}(A)=\left(\frac{1}{\sqrt{2\pi}}\right)^{d} \int_A e^{-\frac{1}{2} \|X\|_F^2} d\rho(X),
\end{align}
where $A \subset \mathfrak{p}_0$ is a Borel set and $\rho=h_\ast \left(\lambda_{ \mathbb{R}^{d}}\right)$ is the pushforward measure of the Lebesgue measure on $ \mathbb{R}^d$ by $h$.

\begin{proposition}\thlabel{gaussian measure conjugation invariant}
    $\gamma_{\mathfrak{p}_0}$ is $K$-conjugation invariant.
\end{proposition}

\begin{proof}
    Let $k \in K$, $P \in \mathfrak{p}_0$, and $\conj_k: G \to G$ conjugation by $k$. Notice that
    \begin{align*}
        &\theta\left(\Ad(k)(P) \right)= d\Theta(d \conj_k(P))= d(\Theta \circ \conj_k) (P) = \\
        &=d(\conj_{\Theta(k)} \circ \, \Theta)(P) = \Ad(\Theta(k)) (\theta(P)),
    \end{align*}
    where $\Ad$ is the adjoint representation of $G$.
    Since $K$ is the set of fixed points of $\Theta$ and $\mathfrak{p}_0$ is the $-1$ eigenspace of $\theta$, it follows that 
    \begin{align*}
        \theta\left(\Ad(k)(P) \right)=\Ad(\Theta(k)) (\theta(P))=-\Ad(k)(P).
    \end{align*}
    Therefore, $Ad(k)(P) \in \mathfrak{p}_0$. By invertibility of $Ad(k)$, we conclude that $Ad(k)\mathfrak{p}_0=\mathfrak{p}_0$. Since we are working with matrix groups, $Ad(k)(P)=kPk^{-1}$. Thus, $\mathfrak{p}_0$ is invariant by $K$-conjugation, so it makes sense to consider if $\gamma_{\mathfrak{p}_0}$ is $K$-conjugation invariant.
    Let $x=(x_1,\dots,x_d) \in  \mathbb{R}^{d}, k \in K$. By definition,
    \begin{align*}
        \gamma_{\mathfrak{p}_0}(A)=\left(\frac{1}{\sqrt{2\pi}}\right)^{d}\int_{h^{-1}(A)}e^{-\frac{1}{2}\|h(x)\|_{F}^{2}} d\lambda_{ \mathbb{R}^{d}}(x).
    \end{align*}
    Hence,
    \begin{align*}
        \gamma_{\mathfrak{p}_0}(k A k^T)=\left(\frac{1}{\sqrt{2\pi}}\right)^{d}\int_{h^{-1}(k A k^T)}e^{-\frac{1}{2}\|h(x)\|_F^2} d\lambda_{ \mathbb{R}^{d}}(x).
    \end{align*}
    By definition of the Frobenius inner product, the map $\operatorname{conj}_k$ given by conjugation by $k$ is a orthogonal automorphism of $\mathfrak{p}_0$. Hence, $h^{-1} \circ \operatorname{conj}_k \circ h$ is a orthogonal automorphism of $\mathbb{R}^{d}$. Consider the change of variables 
    \begin{align*}
        x \mapsto (h^{-1} \circ \operatorname{conj}_k \circ h)(x).
    \end{align*}
    We obtain
    \begin{align*}
        \gamma_{\mathfrak{p}_0}(k A k^T) &=\left(\frac{1}{\sqrt{2\pi}}\right)^{d}\int_{h^{-1}(k A k^T)}e^{-\frac{1}{2}\|h(x)\|_F^2} d\lambda_{ \mathbb{R}^{d}}(x) = \\
        &=\left(\frac{1}{\sqrt{2\pi}}\right)^{d}\int_{h^{-1}(A)}e^{-\frac{1}{2}\|h(h^{-1} \circ \,\operatorname{conj}_k \circ \, h) (x)\|_F^2} d\lambda_{ \mathbb{R}^{d}}((h^{-1} \circ \operatorname{conj}_k \circ \, h)(x)) = \\
        &=\left(\frac{1}{\sqrt{2\pi}}\right)^{d}\int_{h^{-1}(A)}e^{-\frac{1}{2}\|h(x)\|_F^2} d\lambda_{ \mathbb{R}^{d}}(x) = \\
        &=\gamma_{\mathfrak{p}_0}(A).
    \end{align*}
    In the first step, we applied the change of variables. In the second step, we used the invariance of the Lebesgue measure by orthogonal transformations and the $O(n)$-invariance of the Frobenius norm.
\end{proof}

\subsection{Separation by hyperplanes}
We will also make use of the next lemma in the proof of \thref{Main}. Although well-known in the literature, to the best of our knowledge, there is no written proof of it, so we give one here. The terminology being used is the following: an hyperplane of $V$ will be written in the form
\begin{align*}
    H_{u, \alpha}=\left\{x \in \mathbb{R}^n:\langle x, u\rangle=\alpha, \ \alpha \in \mathbb{R}\right\}
\end{align*}
where $u \in V \setminus \{0\}$ is the normal vector of $H_{u,\alpha}$. The hyperplane $H_{u,\alpha}$ bounds the two closed half-spaces
\begin{align*}
    H_{u, \alpha}^{-}:=\left\{x \in \mathbb{R}^n:\langle x, u\rangle \leq \alpha\right\} \quad \mbox{and} \quad  H_{u, \alpha}^{+}:=\left\{x \in \mathbb{R}^n:\langle x, u\rangle \geq \alpha\right\} .
\end{align*}
In case there is no confusion, we will simply denote an hyperplane and its half-spaces by $H,H^-,H^+$, respectively. For some subset $A \subset V$, we say $H$ is a supporting hyperplane of $A$ if there exists $x \in A \cap H$ and, furthermore, either $A \subset H^-$ or $A \subset H^+$. We say that the hyperplane $H_{u,\alpha}$ separates two sets $A$ and $B$ if $A \subset H_{u,\alpha}^-$ and $B \subset H_{u,\alpha}^+$ or vice-versa. They are strongly separated by $H_{u,\alpha}$ if there exists an $\varepsilon >0$ such that $H_{u,\alpha-\varepsilon}$ and $H_{u,\alpha +\varepsilon}$ both separate $A$ and $B$.
\begin{lemma}\thlabel{weakly separated iff}
	Let $M,L \in \mathcal{K}\left( V \right) $, $G$ a subgroup of $GL\left( V \right) $ and $\overline{G}(M, L)$ the set of elements of $\overline{G}$ for which $M \cap \left(\overline{g}L\right) \neq \emptyset$, but such that they can be separated by a hyperplane. If we write $\overline{g}L=gL+t$, then we have that $\overline{g} \in \overline{G}(M,L)$ if and only if $t \in \partial(M - g L)$.
\end{lemma}

\begin{proof}
	If $t \in \partial(M-g L)$, since $M-g L$ is a convex body, by \cite[Theorem 1.3.2]{schneider2013convex} there is a supporting hyperplane of $M - g L$ through $t$. Hence, $M-gL$ and $\{t\}$ are separated by an hyperplane. This is equivalent, by \cite[Lemma 1.3.6]{schneider2013convex}, to $M$ and $gL +t$ being separated by a hyperplane. On the other hand, suppose that $M$ and $gL + t$ are separable by a hyperplane, but $M \cap \left(gL+t\right) \neq \emptyset$. Again, \cite[Lemma 1.3.6]{schneider2013convex} implies that $M- g L$ and $\{t\}$ are separable by a hyperplane. Now suppose that $t \notin \partial(M-g L)$. If $\dim(M-gL)<n$, then $t \notin M-gL$, so $\{t\}$ and $M-gL$ can be strongly separated, hence $K$ and $gL+t$ can be strongly separated, so $M \cap \left(gL+t\right) = \emptyset$, a contradiction. If $\dim(M-gL)=n$, we first show that $t \notin \operatorname{int}(M-gL)$. Suppose, by contradiction,  that   $t \in \operatorname{int}(M-gL)$. Then there would exist a $\varepsilon>0$ such that 
    \begin{align*}
        t+\varepsilon B^n \subset M - gL.     
    \end{align*}
    Let $H_{u,\alpha}$ be the hyperplane that separates $M$ and $gL+t$. Without loss of generality, we have
    \begin{align*}
        M \subset H^{-}_{u,\alpha}, \quad gL+t \subset H^+_{u,\alpha}.
    \end{align*}
    For any $v \in B^n$, we have $t+\varepsilon v \in M-gL$. Then there exist $m\in M,l \in L$ such that
    \begin{align*}
        t+\varepsilon v= m-gl \iff m=gl+t+\varepsilon v.
    \end{align*}
    Therefore,
    \begin{align*}
        \alpha + \varepsilon \langle v,u \rangle\leq \langle gl+t+\varepsilon v, u \rangle=\langle m,u \rangle \leq \alpha.
    \end{align*}
    Hence, we have $\langle v, u \rangle \leq 0$ for all $v \in B^n$. Since $v \in B^n \iff -v \in B^n$, we must have $\langle v,u \rangle = 0$ for all $v$, that is, $u=0$, so there is no separating hyperplane, a contradiction. Hence, the other possible case is $t \notin M-gL$, which gives the same contradiction as above.
\end{proof}
\subsection{Intrinsic volumes} We remind the reader how to compute the intrisic volumes of the unit ball $B^{n}$, which will be used in Section \ref{O(n) case}. Letting $M=B^{n}$ in Steiner's formula \eqref{Steiner's Formula}, we get 
\begin{align*}
    \sum_{j=0}^{n} \epsilon^{n-j} \kappa_{n-j} V_{j}\left(B^{n}\right)=V_{n}\left(B^{n}+\epsilon B^{n}\right)=(1+\epsilon)^{n} \kappa_{n}=\sum_{j=0}^{n} \epsilon^{n-j} {\binom{n}{j}} \kappa_{n}.
\end{align*}

Seeing the expression as a polynomial in $\varepsilon$, we obtain 
\begin{align}\label{intrinsic volumes of unit ball}
	V_j(B^n)={\binom{n}{j}} \frac{\kappa_n}{\kappa_{n-j}}.
\end{align}

\subsection{Norm on \texorpdfstring{$\Val(V)$}{ValV}}Finally, we mention that, with the topology of uniform convergence on compact sets, the space of valuations $\Val(V)$ is a Fréchet space. Moreover, \thref{McMullen's Theorem} implies that the semi-norm 
\begin{align}
	\|\varphi\|:= \sup_{ K \in \mathcal{K}\left( B^{n} \right) }  |\varphi\left( K \right) | \label{valuation norm}
\end{align}
is in fact a norm, making $\Val(V)$ a Banach space.

\section{Proof of Theorem \ref{Main}}\label{Proof}

\begin{proof}
    First we show that $\phi(M \cap \overline{g}L)$ is integrable on $\overline{G}$, showing that the left hand side of \eqref{Main formula} is well defined. 
  Note that  $\overline{g}L= gL +t = k e^X L + t$ where $(k, X,t)  \in K \times \mathfrak{p}_0 \times \mathbb{R}^n$ and,  by definition of $m_{\overline{G}}$, we have
\begin{align}\label{main integral}
    \int_{\overline{G}} \phi( M \cap \overline{g} L) \, dm_{\overline{G}} (\overline{g})
        =\int_{K} \int_{\mathfrak{p}_0} \int_{V}\phi(M \cap \left(k e^X L + t\right)) d\lambda_{V}(t) d\gamma_{\mathfrak{p}_0}(X) d\nu_{K}(k).
\end{align}

    Let $\overline{G}(M,L)$ be the set of triples $(k,X,t) \in K \times \mathfrak{p}_0 \times \mathbb{R}^n$ such that $M \cap k e^X L + t \neq \emptyset$, but can be separated by a hyperplane. By \thref{weakly separated iff}, we have that $(k, X,t) \in \overline{G}(M,L)$ iff $t \in \partial(M-k e^X L)$. Hence we get
    \begin{align*}
        &\nu_{K} \times \gamma_{\mathfrak{p}_0} \times \lambda_{V} \left(\overline{G}(M,L) \right) = \\
        &\int_{K} \int_{\mathfrak{p}_0} \int_{V} \mathbf{1}_{\overline{G}\left( M,L \right) }(k, X, t) d\lambda_{V}(t) d\gamma_{\mathfrak{p}_0}(X) d\nu_{K}(k) = \\
        &\int_{K} \int_{\mathfrak{p}_0} \lambda_{V}(\partial(M-k e^X L)) d\gamma_{\mathfrak{p}_0}(X) d\nu_{K}(k) = 0.
    \end{align*}
    For $(k,X,t) \in K \times \mathfrak{p}_0 \times \mathbb{R}^n \setminus \overline{G}(M,L)$, the function \[
    (k,X,t) \mapsto \Phi(M \cap \left(k e^X L + t\right))
    \] 
    is continuous, being the composition of the continuous functions \[
    (k,X,t) \mapsto k e^X L +t, \quad L \mapsto M \cap L, \quad M \to \phi\left( M \right) 
    ,\] 
	the middle one being continuous by \cite[Theorem 1.8.10]{schneider2013convex}. Therefore,\[
	(k,X,t) \mapsto \Phi(M \cap \left(k e^X L + t \right))
	.\] 
	is continuous outside a set of measure 0, hence $\nu_{K} \times \gamma_{\mathfrak{p}_0} \times \lambda_{V}$-measurable. Furthermore, $\phi$ is bounded on the compact set $\{L \in \mathcal{K}\left( V \right) : L \subset M\}$, so to check integrability of $\phi(M \cap \overline{g}L)$ it is enough to show that \[
	\nu \times \gamma_{\mathfrak{p}_0} \times \lambda_{V}(\{(k,X,t):\phi(M \cap \left(k e^X L +t\right)) \neq 0\}) < \infty
	.\]  
	We have that
    \begin{align*}
        &\nu_{K} \times \gamma_{\mathfrak{p}_0} \times \lambda_V(\{(k,X,t):\phi(M \cap (k e^X L + t)) \neq 0\}) \leq \\
        &\leq \nu_{K} \times \gamma_{\mathfrak{p}_0} \times \lambda_V(\{(k,X,t):M \cap (k e^X L+t) \neq \emptyset\}) = \\
        &=\int_{K} \int_{\mathfrak{p}_0} \int_{V} \mathbf{1}_{\{(k,X,t): M \cap (k e^X L + t) \neq \emptyset\}}(k,X,t) d\lambda_V(t)d\gamma_{\mathfrak{p}_0}(X)d\nu_{K}(k).
    \end{align*}
    By compactness, both $M$ and $L$ are contained in a cube $Q=[-r,r] \times \hdots \times [-r,r]= [-r,r]^n$ for some large enough $r$. Since multiplication by $k$ and $e^{X}$, and translation by $t$ are all bijections, it follows that 
    \[
    M \cap (k e^X L + t) \subset Q \cap (k e^X Q + t) 
    .\] 
    Furthermore, $Q$ is contained in a ball centered at $0$ of radius $\frac{\diam(Q)}{2}$, that is, $ Q \subset \frac{\diam(Q)}{2}B^n$, where $\diam(C)=\sup_{x,y \in C}\|x-y\|.$ Likewise, \[
    k e^X Q \subset \frac{\diam(k e^X Q)}{2} B^n=\frac{\diam(e^X Q)}{2} B^n
    ,\]  
    since $k \in K \subset O(n)$. 
    We have that 
    $$\frac{\diam(Q)}{2}B^n \cap \left(\frac{\diam(e^X Q)}{2} B^n + t \right) \neq \emptyset \iff \|t\| \leq \frac{\diam(Q)+\diam(e^X Q)}{2}.$$
    Adding it all up, we get
    \begin{align*}
        &\int_{K} \int_{\mathfrak{p}_0} \int_{V} \mathbf{1}_{\{(k,X,t): M \cap (k e^X L + t)\neq \emptyset\}}(k,X,t) d\lambda_V(t)d\gamma_{\mathfrak{p}_0}(X)d\nu_K(k) \leq \\
        &\leq \int_{K} \int_{\mathfrak{p}_0} \int_{V} \mathbf{1}_{\{(k,X,t): Q \cap (k e^X Q + t) \neq \emptyset\}}(k,X,t) d\lambda_V(t)d\gamma_{\mathfrak{p}_0}(X)d\nu_K(k) \leq \\
        &\leq \int_{K} \int_{\mathfrak{p}_0} \lambda_V\left(\frac{\diam(Q)+\diam(e^X Q)}{2} B^n \right) d\gamma_{\mathfrak{p}_0}(X) d\nu_{K}(k) = \\
        &=\lambda_V(B^n) \int_{\mathfrak{p}_0} \left(\frac{\diam(Q)+\diam(e^X Q)}{2}\right)^n d\gamma_{\mathfrak{p}_0}(X).
    \end{align*}
    Consider the unit cube $[0,1]^{n}$. We have that 
    \begin{equation}\label{diam}
        \diam(e^X [0,1]^{n})\leq \left\| e^X \begin{bmatrix}
            1 \\
            0 \\
            \vdots  \\
            0
            \end{bmatrix} \right\| 
            + \left\| e^X \begin{bmatrix}
                0 \\
                1 \\
                \vdots  \\
                0
                \end{bmatrix} \right\| 
            + \cdots
            + 
            \left\| e^X \begin{bmatrix}
                0 \\
                0 \\
                \vdots  \\
                1
                \end{bmatrix} \right\| \\
        \leq n \left\|e^X\right\|_2,
    \end{equation}
    where $\left\|e^X\right\|_2=\max_{\|x \| = 1} \left\|e^X x \right\|$ is the matrix norm induced by the euclidean norm $\|\cdot\|$. Let $x \in \mathbb{R}^n$ with $\|x\|=1$ such that $\left\|e^X\right\|_2=\left\|e^X x \right\|$. By completing $\{x\}$ to form an orthonormal basis of $\mathbb{R}^n$, there is $P \in O\left( n \right) $ whose first column is $x$. Then 
    \begin{align}\label{norms1}
        \|e^X\|_2=\|e^X P e_1\|=\sqrt{\|e^X P e_1\|^2}\leq\sqrt{ \sum_{i=1}^n\| e^X P e_i \|^2}= \|e^X P\|_F = \|e^X\|_F.
    \end{align}
    Also, by the submultiplicativity of a matrix norm,
    \begin{align}\label{norms2}
        \left\|e^X \right\| = \left\|\sum_{k=0}^{\infty} \frac{X^k}{k!} \right\| \leq \sum_{k=0}^{\infty} \left\|\frac{X^k}{k!} \right\| \leq \sum_{k=0}^{\infty} \frac{\left\| X \right\|^k}{k!} = e^{\|X\|}. 
    \end{align}
    Combining \eqref{diam}, \eqref{norms1} and \eqref{norms2} we obtain
    \begin{align}
        \diam(e^X [0,1]^{n}) \leq n \left\|e^X\right\|_2 \leq n \left\|e^X \right\|_F \leq n e^{\|X \|_F}.
    \end{align}
Recall that the Minkowski sum of two sets $A, B \subset V$ is given by 
\begin{equation}\label{Minkowski sum}
    A+B= \left\{a+b: a \in A, b \in B \right\}.
\end{equation}
    Using this sum  we can write   $Q=\left[ -r,r \right]^{n}= 2r [0,1]^{n} - (r,\dots,r)$ and it  it follows that 
    \begin{align*} 
        \diam(e^X Q) & = \diam(e^X(2r [0,1]^{n} - (r,\dots,r))) \\
        &=\diam(2r e^X [0,1]^{n} - e^X (r,\dots,r)) = \\
        &=2r \diam(e^X [0,1]^{n}) \leq 2nr e^{\left\|X \right\|_F},
    \end{align*}
since $\diam$ is linear and invariant by translations. 
Therefore we obtain 
   \begin{equation}\label{inequality of diameter}
    \diam(e^X Q) \leq 2nr e^{\left\|X \right\|_F}.
   \end{equation} 
    Therefore, 
    \begin{align*}
        &\int_{\mathfrak{p}_0} \left(\diam(Q) + \diam(e^X Q) \right)^{n} d\gamma_{\mathfrak{p}_0}(X) \\
        &=\left(\frac{1}{\sqrt{2\pi}} \right)^{d}\int_{\mathfrak{p}_0} \left(\diam(Q) + \diam(e^X Q) \right)^{n} e^{-\frac{1}{2}\left\|X \right\|_F^2} d\rho(X) \\
        &\leq \left(\frac{1}{\sqrt{2\pi}} \right)^{d} \int_{\mathfrak{p}_0} \left(\diam(Q) + 2nr e^{\left\| X \right\|_F}\right)^{n} e^{-\frac{1}{2}\left\| X \right\|_F^2} d\rho(X).
    \end{align*}
    Notice that
    \begin{align*}
         & \left(\diam(Q) + 2nr e^{\left\| X \right\|_F}\right)^{n} e^{-\frac{1}{2}\left\| X \right\|_F^2} \\
          & =\left(\sum_{i=0}^{n} \binom{n}{i} \left(\diam(Q)\right)^i\left(2nr e^{\|X\|_F}\right)^{n-i}\right) e^{-\frac{1}{2}\left\| X \right\|_F^2} = \\
        &=\sum_{i=0}^{n} \binom{n}{i}\left(\diam(Q)\right)^i\left(2nr\right)^{n-i} e^{(n-i)\|X\|_F-\frac{1}{2}\left\|X \right\|_F^2},
    \end{align*}
    so we can conclude that integral \eqref{main integral} is finite. Hence, applying Fubini's Theorem and \thref{Classical General kinematic formula}, we have
    \begin{align*}
        &\int_{K} \int_{\mathfrak{p}_0} \int_{V}\phi(M \cap \left(k e^X L + t\right)) d\lambda_V(t) d\gamma_{\mathfrak{p}_0}(X) d\nu_{K}(k) = \\
        &=\int_{\overline{K}} \int_{\mathfrak{p}_0} \phi(M \cap \overline{k}e^X L) d\gamma_{\mathfrak{p}_0}(X) d\mu_{\overline{K}}(\overline{k})= \sum_{i=1}^{N} \sum_{j=1}^{N} c_{ij}^{\phi} \psi_i(M) \int_{\mathfrak{p}_0} \psi_j(e^X L) d\gamma_{\mathfrak{p}_0}(X).
    \end{align*}
    We can simplify $\Psi_j(L):=\int_{\mathfrak{p}_0} \psi_j(e^X L) d\gamma_{\mathfrak{p}_0}(X)$ further by showing that $\Psi_j \in \Val^{K}\left( V \right) $. We have that
    \begin{align*}
        &\Psi_j(M \cup L) + \Psi_j(M \cap L)= \int_{\mathfrak{p}_0} \psi_j(e^X \left(M \cup L\right)) + \psi_j(e^X \left(M \cap L\right)) \, d\gamma_{\mathfrak{p}_0}(X) = \\
        &= \int_{\mathfrak{p}_0} \psi_j(e^X M \cup e^X L) + \psi_j(e^X M \cap e^X L) \, d\gamma_{\mathfrak{p}_0}(X) = \\
        &=\int_{\mathfrak{p}_0} \psi_j(e^X M) + \psi_j(e^X L) \, d\gamma_{\mathfrak{p}_0}(X) = \\
        &=\Psi_j(M) + \Psi_j(L),
    \end{align*}
    so $\Psi_j$ is a valuation. 
    Now let $k \in K$
    \begin{align*}
        \Psi_j(k L)& = \int_{\mathfrak{p}_0} \psi_j(e^X k L) d\gamma_{\mathfrak{p}_0}(X)  = \int_{\mathfrak{p}_0} \psi_j(k k^{-1} e^X k L) d\gamma_{\mathfrak{p}_0}(X)  \\
        & = \int_{\mathfrak{p}_0} \psi_j(k e^{k^{-1}X k} L) d\gamma_{\mathfrak{p}_0}(X)  =\int_{\mathfrak{p}_0} \psi_j(e^{k^{-1}X k} L) d\gamma_{\mathfrak{p}_0}(X), 
    \end{align*}
    where the last equality follows from $K$-invariance of $\psi_j$. 
    Then, by \thref{gaussian measure conjugation invariant} we obtain,
    \begin{equation*}
        \int_{\mathfrak{p}_0} \psi_j(e^{k^{-1}X k} L) d\gamma_{\mathfrak{p}_0}(X) = \int_{\mathfrak{p}_0} \psi_j(e^{X} L) d\gamma_{\mathfrak{p}_0}(X) =\Psi_j(L).
    \end{equation*}
    Hence, $\Psi_j$ is $K$-invariant. 
    
    Next, we show that $\Psi_j$ is continuous. Suppose that $L_i \to L$ in the Hausdorff metric. Then $e^X L_i \to e^X L$, which implies $\psi_j(e^X L_i) \to \psi_j(e^X L)$. Let $r>0$ be large enough so that $L_i,L \subset [-r,r]^n$ for all $i$. By ~\eqref{inequality of diameter}, we have 
    \[
    e^{X}L_i  \subset  e^{X}\left[ -r,r \right]^{n} \subset 2nr e^{\|X\|_{F}}B^{n}
    .\]
    And we have a similar result for $e^{X}L$. Therefore, it follows from  \ref{valuation norm}that
    \[
    \Psi_j\left( \frac{e^{X}L_i}{2nre^{\|X\|_{F}}} \right) \le \|\Psi_j\| \iff \Psi_j\left( e^{X}L_i \right) \le \left( 2nre^{\|X\|_{F}} \right)^{\deg\left( \Psi_j \right) } \|\Psi_j\|
    ,\]
    and we obtain a similar inequality for $\Psi_j\left( e^{X}L \right) $. Moreover, 
    \begin{align*}
    	\int_{\mathfrak{p}_0}& \left( 2nre^{\|X\|_{F}} \right)^{\deg\left( \Psi_j \right) } \|\Psi_j\| d\gamma_{\mathfrak{p}_0}\left( X \right)  = \\
	 &\left( 2nr \right)^{\deg\left( \Psi_j \right) } \|\Psi_j\| \int_{\mathfrak{p}_0} e^{\deg\left( \Psi_j \right) \|X\|_{F}} e^{-\frac{\|X\|_{F}^{2}}{2}} d\rho\left( X \right) < \infty.
    \end{align*}
    Hence, we conclude that 
    $$ \int_{\mathfrak{p}_0} \psi_j(e^{X} L_i) d\gamma_{\mathfrak{p}_0}(X) < \infty  \quad \mbox{and}  \quad \int_{\mathfrak{p}_0} \psi_j(e^{X} L) d\gamma_{\mathfrak{p}_0}(X) < \infty, $$
    and  the conditions of the dominated convergence theorem are satisfied. It follows that 
    \begin{align*}
        \Psi_j(L_i)=\int_{\mathfrak{p}_0} \psi_j(e^{X} L_i) d\gamma_{\mathfrak{p}_0}(X) \to \int_{\mathfrak{p}_0} \psi_j(e^{X} L) d\gamma_{\mathfrak{p}_0}(X) = \Psi_j(L),
    \end{align*} 
    so $\Psi_j \in \Val^{K}\left( V \right) $. Moreover, since $\psi_j$ is homogeneous of degree $\deg\left( \psi_{j}\right) $, so is $\Psi_j$, that is $\Psi_j \in Val_{\deg\left( \psi_j \right) }^K$. Therefore, since $\dim \Val^K (V) < \infty$ then it follows from  McMullen's \thref{McMullen's Theorem} that there are $c\left( \psi_l \right)  \in \mathbb{R}$ such that 
    \begin{align*}
	  \Psi_j(L)=  \int_{\mathfrak{p}_0} \psi_j(e^{X} L) d\gamma_{\mathfrak{p}_0}(X) = \sum_{\psi_l \in Val_{\deg\left( \psi_j \right) }^{K}} c\left( \psi_l \right)  \psi_l(L).
    \end{align*}
\end{proof}

\section{The case \texorpdfstring{$K=O(n)$}{KO(n)}}\label{O(n) case}
	In the particular case when $K=O(n)$, the space of valuations $\Val^{O(n)}$ is spanned by the intrinsic volumes $V_i$, with $0\le i\le n$, which simplifies the result of \thref{Main}. In fact, we can prove something stronger, generalizing \thref{Hadwiger's General Integral Geometric Theorem}.
	\begin{proof}[Proof of Theorem \ref{thm: O(n) case}]
		Most of the proof is the same as in \thref{Main}. Note that in  this case $\mathfrak{p}_0= \Sym(n)$  is the space of symmetric matrices. The only difference is after proving that the integral is finite. Now we have
		\begin{align*}
		\int_{\overline{GL(n)}} & \varphi( M \cap \overline{g} L) \, dm_{\overline{GL(n)}} (\overline{g}) = \\  & = \int_{O(n)}  \int_{\Sym(n)} \int_{V}\varphi(M \cap \left(\vartheta e^X L + t\right)) d\lambda_{V}(t) d\gamma_{\Sym(n)}(X) d\nu_{O(n)}(\vartheta) = \\
		&=\int_{\overline{O(n)}} \int_{\Sym(n)} \varphi(M \cap \overline{\vartheta}e^X L) d\gamma_{\Sym(n)}(X) d\mu_{\overline{O(n)}}(\overline{\vartheta}) \\ & = 2\sum_{k=0}^{n}\varphi_{n-k}(K)\int_{\Sym(n)} V_k(e^X L) d\gamma_{\Sym(n)}(X),
		\end{align*}
		where we have used Hadwiger's General Integral Geometric  \thref{Hadwiger's General Integral Geometric Theorem}. 
        \begin{remark}\label{2 components}
        Note that the factor of 2 appears because \( O(n) \) has two connected components: one is \( SO(n) \), and the other is \( O^-(n) \), the set of matrices in \( O(n) \) with determinant \( -1 \). Since \( -I_n \in O(n) \) and \( SO(n) = -I_n O^-(n) \), and given the \( O(n) \)-invariance of the Haar measure \( \nu_{O(n)} \), it follows that \( \nu_{O(n)}(SO(n)) = \nu_{O(n)}(O^-(n)) \).
        \end{remark}
        Next, the same reasoning as in the proof of \thref{Main} shows that \[
		\Psi_j (M) =\int_{\Sym(n)} V_j(e^X M) d\gamma_{\Sym(n)}(X) \in \Val_{j}^{O(n)} 
		.\] 
		In this case, since $\Val_{j}^{O(n)}=\left<V_{j} \right>$, there is $c_j$ such that 
        \begin{equation} \label{ck}
		\int_{ \Sym(n)} V_j\left( e^{X} M \right) d\gamma_{\Sym(n)}\left( X \right) =c_j V_j\left( M \right). 
		\end{equation}
		Letting $M=B^{n}$ and recalling \eqref{intrinsic volumes of unit ball}, we obtain \[
		c_j=\frac{\int_{\Sym(n)} V_j\left( e^{X} B^n \right) d\gamma_{\Sym(n)}\left(X \right)}{\binom{n}{j}\frac{\kappa_n}{\kappa_{n-j}}}
		.\] 
	\end{proof}
	The advantage in this case is that it is possible to explicitly compute the $c_j$'s. For $j=n$ this is straightforward computation by noticing that replacing $M$ by $\left[ 0,1 \right]^{n} $ in \eqref{ck} we obtain 
	$$c_n=\int_{\Sym(n)} V_{n}\left( e^X \left[ 0,1 \right]^{n} \right) d\gamma_{\Sym(n)}\left( X \right).$$
	\begin{lemma}
    Let $\Delta_{ij}$ be the $n \times n$ matrix whose $(i,j)$-th entry is 1 and 0 elsewhere. Then the set
    $$\left\{\left(\Delta_{ii}\right)_{1 \leq i \leq n} ,\frac{1}{\sqrt{2}} \left(\Delta_{ij} + \Delta_{ji}\right)_{1 \leq i < j \leq n}\right\}$$ is an orthonormal basis of $\Sym(n)$.
	\end{lemma}

\begin{proof}
    Since this set has $n+\frac{n(n-1)}{2}=\frac{n(n+1)}{2}=\dim(\Sym(n))$ elements, it is enough to check linear independence. This is just computations, repeatedly using that 
    \begin{align*}
        \Delta_{ij} \Delta_{kl} =
        \left\{ \begin{array}{cl}
              \Delta_{il},& j=k \\
            0, & j\neq k
            \end{array} \right..
    \end{align*}
\end{proof}

    Therefore
	\begin{align*}
		&\int_{\Sym(n)} V_n(e^X [0,1]^n) d\gamma_{\Sym(n)}(X)=\int_{\Sym(n)} \det(e^X) d\gamma_{\Sym(n)}(X)\\ & =  \int_{\Sym(n)} e^{\tr(X)} d\gamma_{\Sym(n)}(X)  \\
		& = \left(\frac{1}{\sqrt{2\pi}} \right)^{\frac{n(n+1)}{2}} \int_{\Sym(n)} e^{\tr(X)-\frac{1}{2}\|X\|^2_F} d\rho(X)  \\
		& = \left(\frac{1}{\sqrt{2\pi}} \right)^{\frac{n(n+1)}{2}} \int_{\mathbb{R}^{\frac{n(n+1)}{2}}} e^{\tr\left(\sum a_i \Delta_{ii}+\sum b_{jk}\left(\frac{\Delta_{jk}+\Delta_{kj}}{\sqrt{2}}\right)\right)-\frac{1}{2}\|\sum a_i \Delta_{ii}+\sum b_{jk}\left(\frac{\Delta_{jk}+\Delta_{kj}}{\sqrt{2}}\right)\|^2_F} d\lambda_{\mathbb{R}^{\frac{n(n+1)}{2}}}(a,b)  \\
		&=\left(\frac{1}{\sqrt{2\pi}} \right)^{\frac{n(n+1)}{2}} \int_{\mathbb{R}^{\frac{n(n+1)}{2}}} e^{\sum \left( a_i -\frac{(a_i)^2}{2} \right) -\sum \frac{(b_{jk})^2}{2}} d\lambda_{\frac{n(n+1)}{2}}(a,b)  \\
		&=\left(\frac{1}{\sqrt{2\pi}} \right)^{\frac{n(n+1)}{2}} \left( \sqrt{2e\pi} \right)^n (\sqrt{2\pi})^{\frac{n(n+1)}{2}-n} =e^{\frac{n}{2}},
	\end{align*}
	where in the  step before the last we used the identity $\int_{\mathbb{R}}e^{-x^{2}}dx=\sqrt{\pi} $. For general $j$, since $X$ is symmetric, it has a diagonalization $X=\vartheta^{T} D(X) \vartheta$. Then, by the $O(n)$ invariance of $V_k$ and of $B^n$,
    \[	V_j\left( e^{X}B^{n} \right) =V_j\left( e^{\vartheta^{T} D\left( X \right) \vartheta } B^{n} \right)=V_j\left( e^{D\left( X \right) }B^{n} \right)   	.\] In other words, \[
	X \to V_k\left( e^{X}B^{n} \right) 
.\] is $O(n)$-conjugation invariant. By a Weyl integration formula \cite{nicolaescu2014complexity},
	\begin{equation}\label{Weil integration formula}
    \begin{aligned}
		\int_{\Sym(n)}  & V_j(e^X B^n) d\gamma_{\Sym(n)}(X) = \\
		&=\frac{1}{Z_n} \int_{\mathbb{R}^n} V_j(e^{diag(\lambda_1,\dots,\lambda_n)} B^n) e^{-\frac{1}{2} \sum_{l=1}^{n} \lambda^2_l} \prod_{l<j}|\lambda_l-\lambda_j| d\lambda_1 \dots d\lambda_n = \\
		&=\frac{1}{Z_n} \int_{\mathbb{R}^n} V_j(\mathcal{E}(e^{\lambda_1},\dots,e^{\lambda_n})) e^{-\frac{1}{2} \sum_{l=1}^{n} \lambda^2_l} \prod_{l<j}|\lambda_l-\lambda_j| d\lambda_1 \dots d\lambda_n,
	\end{aligned}
    \end{equation}
	where $Z_n=2^{\frac{n}{2}}n! \prod_{l=1}^{n} \Gamma\left( \frac{l}{2} \right)$, where $\Gamma$ is the Gamma function and 
    $$\mathcal{E}\left( a_1,\ldots,a_n \right) = \left\{  (x_1, \hdots,  x_n ) \in \mathbb{R}^n : \left(\frac{x_1}{a_1}\right)^2 + \hdots + \left(\frac{x_n}{a_n}\right)^2 \leq 1 \right\}$$
    is the ellipsoid with semiaxes $a_1,\ldots,a_n > 0$.  On the other hand, in \cite{gusakova2022intrinsic} it was shown that 
    \begin{equation}\label{volume elippsoid}
	V_j(\mathcal{E}(a_1,\dots,a_n))=\kappa_j \sum_{i=1}^d a_i^2 s_{j-1}\left(a_1^2, \ldots, a_{i-1}^2, a_{i+1}^2, \ldots, a_d^2\right) I(a_i)
	,
    \end{equation}
    where
    \[
    I(a_i)= \int_0^{\infty} \frac{t^{j-1}}{\left(a_i^2 t^2+1\right) \prod_{l=1}^d \sqrt{a_l^2 t^2+1}} dt
    \]
    and $s_{j-1}$ is the $\left( j-1 \right) $-th elementary symmetric polynomial. In our case, \[
	V_j(\mathcal{E}(e^{\lambda_1},\dots,e^{\lambda_n}))=\kappa_j \sum_{i=1}^d e^{2\lambda_i} s_{j-1}\left(e^{2\lambda_1}, \ldots, e^{2\lambda_{i-1}}, e^{2\lambda_{i+1}}, \ldots, e^{2\lambda_{n}}\right) I(e^{\lambda_i})
	.\] 
	Combining \eqref{Weil integration formula} and \eqref{volume elippsoid}, we obtain the formula for $c_j$ in terms of integrals in $\mathbb{R}$ and $\mathbb{R}^{n}$.

\printbibliography

\end{document}